\documentclass[twoside,11pt,notcite,notref]{article}
\usepackage{amsthm,amsmath,amsfonts,amssymb,bm,bbm,mathtools} 
\usepackage{graphicx}
\usepackage{tocloft} % For controling spacing in table of contents

\usepackage{datetime} \usdate
\usepackage{currfile} 
\usepackage{color}

\usepackage{textcomp}
\usepackage{hyperref}

\date{\today}
\newif\ifdraft
\ifdraft
\date{DRAFT: \today}
\fi

\newcommand{\myrunningheads}
{\ifdraft
\pagestyle{myheadings}
\markboth
{Filename: \currfilename \quad DRAFT \mmddyyyydate\today\quad \currenttime }
{Filename: \currfilename \quad DRAFT \mmddyyyydate\today\quad \currenttime }
\else
\pagestyle{myheadings}
\markboth
{Self-similar spreading in merging-splitting models} % of animal group size}
{J.-G. Liu, B. Niethammer and R. L. Pego}
\fi
}
\myrunningheads

\usepackage{pdfsync,hyperref}

\newcommand{\nwc}{\newcommand}
\newcommand{\hide}[1]{} % Hides whatever is in braces: replace by {#1} to unhide
\nwc{\PhiNew}{\Phi_{\star}}

\nwc{\iav}{{i_{\rm av}}}
\nwc{\xav}{{x_{\rm av}}}
\nwc{\feq}{{f_{\rm eq}}}
\nwc{\Feq}{{F_{\rm eq}}}
\nwc{\fheq}{{f^h_{\rm eq}}}
\nwc{\fheqi}{{f^h_{\rm eq\,\it i}}}
\nwc{\Fheq}{{F^h_{\rm eq}}}

\nwc{\finit}{f_{\rm in}}
\nwc{\Finit}{F_{\rm in}}
\nwc{\mui}{{\mu\:\! i}}  % subscript mu with half-thin space

\renewcommand{\epsilon}{\varepsilon}
\nwc{\eps}{\epsilon}

\nwc{\ip}[1]{\langle #1 \rangle}
\nwc{\qref}[1]{(\ref{#1})}

\nwc{\rplus}{{{\mathbb R}_+}}
\nwc{\D}{\partial}
\nwc{\inv}{^{-1}}
\nwc{\ubar}{\underline}
\nwc{\uu}{U}

\newcommand{\supn}{^{(n)}}
\nwc{\supk}{^{(k)}}
\nwc{\supj}{^{(j)}}
\nwc{\wkto}{\xrightarrow{w}}
\nwc{\vto}{\xrightarrow{v}}
\nwc{\nto}{\xrightarrow{n}}

\nwc{\np}{{n+1}}
\nwc{\one}{\mathbbm{1}}
\newcommand{\R}{\mathbb{R}}
\newcommand{\C}{\mathbb{C}}

\newcommand{\N}{\mathbb{N}}

\newcommand{\EE}{\mathbb{E}} %\bm{E}}

\newcommand{\calM}{\mathcal{M}}

\newcommand{\calA}{\mathcal{A}}

\def\theequation{\thesection.\arabic{equation}}

\def\paragraph#1{{\bf #1\ }}

\newtheorem{lemma}{Lemma}[section]  

\newtheorem{theorem}[lemma]{Theorem}

\newtheorem{definition}[lemma]{Definition}
\newtheorem{proposition}[lemma]{Proposition}

\theoremstyle{definition}  % The following environments do not use italicS
\newtheorem{remark}{Remark}[section]

\numberwithin{equation}{section}

\makeatletter
\@mparswitchfalse%Put all sidenotes on the right.
\makeatother

\makeatletter
\newcommand{\eqlab}[1]{\leavevmode\hfill\refstepcounter{equation}\label{#1}\textup{\tagform@{\theequation}}}
\makeatother

\newcommand{\warning}[1]{\typeout{}\typeout{WARNING: #1 at line \the\inputlineno}\typeout{}}
\newenvironment{todo}[1][TODO]{%
    \ifdraft\else\warning{TODO still present in final version}\fi
    \MakeFramed{\advance\hsize-\width \FrameRestore}\textbf{#1. }}%
    {\endMakeFramed}
    {\end{todo}}

\DeclareMathOperator{\im}{{\rm Im}}

\def\Box{\leavevmode\vbox{\hrule
     \hbox{\vrule\kern4pt\vbox{\kern4pt}%
           \vrule}\hrule}}

\newcounter{appendix}
\def\appendix{\advance\c@appendix by 1
   \def\thesection{\Alph{section}}
   \ifnum\c@appendix=1 \setcounter{section}{-1} \fi
   \@startsection {section}{1}{\z@}{-3.5ex plus -1ex minus 
   -.2ex}{2.3ex plus .2ex}{\Large\bf}}

\makeatletter
\def\@part[#1]#2{%
  \ifnum \c@secnumdepth >-2\relax
    \refstepcounter{part}%
    \addcontentsline{toc}{part}{\thepart\hspace{1em}#1}%
  \else
    \addcontentsline{toc}{part}{#1}%
  \fi
  \markboth{}{}%
  {\centering
   \interlinepenalty \@M
   \normalfont
   \ifnum \c@secnumdepth >-2\relax
     \LARGE\bfseries \thepart\ \ % 
   \fi
   #2\par}
  }
\makeatother
\nwc{\vp}{\varphi}
\nwc{\blue}[1]{\textcolor{blue}{#1}}
\date{} 
\begin{document}

\title
{Self-similar spreading in a merging-splitting model of animal group size}

\author{Jian-Guo Liu$^{(1)}$, B. Niethammer$^{(2)}$, Robert L. Pego$^{(3)}$}  

\maketitle

\vspace{-0.2 cm}

\begin{center}
1-Department of Physics and Department of Mathematics\\
Duke University,
Durham, NC 27708, USA\\
email: jliu@phy.duke.edu
\end{center}

\begin{center}
2- Institut f\"ur Angewandte Mathematik\\
Universit\"at Bonn\\
Endenicher Allee 60\\
53115 Bonn, Germany\\
email: niethammer@iam.uni-bonn.de
\end{center}

\begin{center}
3-Department of Mathematics
and Center for Nonlinear Analysis\\
Carnegie Mellon University,
Pittsburgh, Pennsylvania, PA 12513, USA\\
email: rpego@cmu.edu
\end{center}

\vspace{-0.2 cm}
\begin{abstract}
In a recent study 
of certain merging-splitting models of animal-group size
(Degond {\it et al.}, {\it J.~Nonl.~Sci.}~27 (2017) 379), 
it was shown that an initial size distribution with infinite first moment
leads to convergence to zero in weak sense, corresponding to 
unbounded growth of group size.  In the present paper we show that 
for any such initial distribution with a power-law tail, 
the solution approaches a self-similar spreading form.  
A one-parameter family of such self-similar solutions exists, 
with densities that are completely monotone, having 
power-law behavior in both small and large size regimes,
with different exponents.

\end{abstract}

\medskip
\noindent
{\bf Key words: } Fish schools, Bernstein functions, complete monotonicity, heavy tails, convergence to equilibrium.

\medskip
\noindent
{\bf AMS Subject classification: }{45J05, 70F45, 92D50, 37L15, 44A10, 35Q99.}

\medskip
\noindent
{\bf Running head:}
{Self-similar spreading in merging-splitting models} 
\vskip 0.4cm

\pagebreak
    
\setlength{\cftbeforesecskip}{4pt}
\setlength{\cftbeforepartskip}{9pt}
\renewcommand{\cftsecfont}{\normalfont}

\ifdraft
\tableofcontents
\fi

%%%%%%%%%%%%%%%%%%%%%%%%%%%%%%%%%%%%%%%%%%%%%%%%%%%%%%%%%%%%%%%%%%%%%%%%%%%%%%%%%%%%%%%%%%%%%%%%
%%%%%%%%%%%%%%%%%%%%%%%%%%%%%%%%%%%%%%%%%%%%%%%%%%%%%%%%%%%%%%%%%%%%%%%%%%%%%%%%%%%%%%%%%%%%%%%%
%%%%%%%%%%%%%%%%%%%%%%%%%%%%%%%%%%%%%%%%%%%%%%%%%%%%%%%%%%%%%%%%%%%%%%%%%%%%%%%%%%%%%%%%%%%%%%%%
%%%%%%%%%%%%%%%%%%%%%%%%%%%%%%%%%%%%%%%%%%%%%%%%%%%%%%%%%%%%%%%%%%%%%%%%%%%%%%%%%%%%%%%%%%%%%%%%

%\typeout{Line Width: \the\linewidth}

%% Text of document
\vfil\pagebreak

%% !TEX root = DLPmain.tex

%%%%%%%%%%%%%%%%%%%%%%%

\setcounter{equation}{0}
\section{Introduction}
\label{intro}

Coagulation-fragmentation equations can be used to describe a large variety of merging and splitting processes, including the evolution of animal group sizes \cite{Ma_etal_JTB11}.
 We refer to \cite{DLP2017} for an extensive discussion of the relevant literature in this particular application area. 

 Here we consider a model with constant coagulation and overall fragmentation rate coefficients that lacks detailed balance and a corresponding $H$-theorem. 
%{\clb 
{This model is motivated by a compelling analysis of fisheries data that was carried out by H.-S. Niwa in \cite{Niwa-JTB2003},
and a first mathematical study of the behavior of its solutions was performed in \cite{DLP2017}.
 As demonstrated in \cite{DLP2017},  the nature of} equilibria of this model as well as their domains of attractions can be rigorously studied
using the theory of Bernstein functions. More precisely, it was shown that equilibria can be expressed by a single smooth scaling profile which is not explicit, but it has a convergent power-series representation and its behaviour for small and large cluster sizes
can be completely characterized by different power laws with exponential cutoff \cite[eq. (1.5)-(1.7)]{DLP2017}. 
Furthermore, if the initial data have finite first moment,  solutions converge to  equilibrium in the large time limit.

In addition, it was also shown that if the initial data have infinite first moment, then solutions converge weakly to zero, which means that clusters grow without bound as time goes to infinity.
Our goal in the present paper is to investigate whether this growth behaviour is described by self-similar solutions.
Indeed, we are going to show that there exists a family of self-similar profiles with completely monotone densities, characterized by different power-law tail behaviours for small and large cluster sizes.
Furthermore, if the cumulative mass distribution of the initial data has power law growth for large
cluster sizes, the corresponding solution converges  to the profile whose mass distribution diverges with the same power-law tail. 

Self-similar solutions with fat tails have recently received quite some attention, in particular in the analysis of coagulation equations, starting with work on models with solvable kernels
\cite{B_eternal,MP2004}. For coagulation equations with non-solvable kernels, existence of self-similar profiles with fat tails has been studied in \cite{NV2013,NTV2016,BNV2016}, but to our knowledge this is the first
time that such solutions are found for a class of  coagulation-fragmentation equations. 

We describe both the discrete- and continuous-size versions of the model in section \ref{sec:CF_general}. 
Our proofs use and extend the methods of complex function theory and in particular Bernstein functions  as developed in \cite{MP2004,MP2008,DLP2017} and we give a brief overview of the main definitions and results
in section \ref{sec:prelim}. Our main results are stated in section \ref{sec:main}, while the remaining sections are devoted to their respective proofs. 

%%%%%%%%%%%%%%%%%%%%%%%%%%%%%%%%%%%%%%%%%%%%%%%%%%%%%%%%%%%%%%%%%%%%%%%%%%%%%%%%%%%%%%%%%%%%%%%%
%%%%%%%%%%%%%%%%%%%%%%%%%%%%%%%%%%%%%%%%%%%%%%%%%%%%%%%%%%%%%%%%%%%%%%%%%%%%%%%%%%%%%%%%%%%%%%%%
%%%%%%%%%%%%%%%%%%%%%%%%%%%%%%%%%%%%%%%%%%%%%%%%%%%%%%%%%%%%%%%%%%%%%%%%%%%%%%%%%%%%%%%%%%%%%%%%
\setcounter{equation}{0}
\section{Coagulation-fragmentation Models D and C}
\label{sec:CF_general}

 In this section we describe both the discrete coagulation-fragmentation equations under study as well as their continuous-size analogue.

\subsection{Discrete-size distributions}
The number density of clusters of size $i$ at time $t$ is denoted by $f_i(t)$.
The size distribution $f(t)=(f_i(t))_{i\in\N}$ evolves according to 
discrete coagulation-fragmentation equations, written in strong form as follows:
\begin{eqnarray}
&&\hspace{-1.5cm}
\frac{\partial f_i}{\partial t}(t) = Q_a(f)_i(t) + Q_b(f)_i(t) ,
\label{eq:CF3_disc}\\
&&\hspace{-1.5cm}
Q_a(f)_i(t) = \frac{1}{2} \sum_{j=1}^{i-1} a_{j , i-j}\, f_j(t)  \, f_{i-j}(t) - \sum_{j=1}^\infty a_{i,j} \, f_i(t) \, f_j(t) , 
\label{eq:CF4_disc} \\
&&\hspace{-1.5cm}
Q_b(f)_i(t) = \sum_{j=1}^\infty b_{i,j} \, f_{i+j}(t) - \frac{1}{2} \sum_{j=1}^{i-1} b_{j , i-j} \, f_i(t) \ .
\label{eq:CF5_disc}
\end{eqnarray}
The terms in $Q_a(f)_i(t)$ describe the gain and loss rate of clusters of size $i$ 
due to aggregation or coagulation, and correspondingly the terms in 
$Q_b(f)_i(t)$ describe the rate of breakup or frag\-mentation.

These equations can be written in the following weak
form, suitable for comparing to the continuous-size analog:
We require that for any bounded test sequence $(\varphi_i)$,
\begin{eqnarray}
&&\hspace{-1cm}
\frac{d}{dt} \sum_{i=1}^\infty  \varphi_i \, f_i(t) =
\frac{1}{2} \sum_{i,j=1}^\infty \big( \varphi_{i+j} - \varphi_i - \varphi_j \big) \,  a_{i,j} \, f_i(t) \, f_j(t) \nonumber \\
&&\hspace{1.5cm}
- \frac{1}{2} \sum_{i=2}^\infty \Big( \sum_{j=1}^{i-1} \big( \varphi_i - \varphi_j - \varphi_{i-j} \big) \,  b_{j , i-j} \, \Big) f_i(t) \,. 
\label{eq:CF2_disc}
\end{eqnarray}

The present study deals with the particular case when 
the rate coefficients take the form
\begin{eqnarray}
a_{i,j} = \alpha\,, \qquad b_{i,j} = \frac{\beta}{i+j+1} \,,
\qquad \alpha=\beta=2.
\label{eq:rates_niwa_discD} 
\end{eqnarray}
We refer to the coagulation-fragmentation equations 
\qref{eq:CF3_disc}-\qref{eq:CF5_disc} with the coefficients 
in \qref{eq:rates_niwa_discD} as \textbf{Model D} (D for discrete size).
By a simple scaling we can achieve any values of $\alpha,\beta>0$ 
and so we keep $\alpha=\beta=2$ for simplicity.
As discussed in \cite{DLP2017}, Model D arises as a modification of 
the time-discrete model
written in \cite{Ma_etal_JTB11} which essentially corresponds to the choice
of rate coefficients as
 \begin{eqnarray}
a_{i,j} = \alpha\,, \qquad b_{i,j} = \frac{\beta}{i+j-1}\,.
\label{eq:rates_niwa_disc} 
\end{eqnarray}
These choices correspond to taking the rate that pairs of individual clusters 
coalesce,
and the rate that individual clusters fragment,
to be constants independent of size. 

The modification in \eqref{eq:rates_niwa_discD}, however, permits an analysis in terms of the
{\sl Bernstein transform} of the size-distribution measure
 $\sum_{j=1}^\infty f_j(t)\,\delta_j(dx)$. 
This Bernstein transform is given by 
\begin{equation}\label{eqD:BTdef}
\breve f(\hat s,t) = \sum_{j=1}^\infty (1-e^{-j\hat s}) f_j(t) \ .
\end{equation}
Taking $\varphi_j = 1-e^{-j\hat s}$ in \qref{eq:CF2_disc},
it {can be shown (see \cite[Eq.(10.6)]{DLP2017})}  that $\breve f(\hat s,t)$ satisfies the integro-differential equation
\begin{equation}\label{eqD:BT1}
\D_t \breve f(\hat s,t) = -\breve f^2 - \breve f + 
  \frac2{1-e^{-\hat s}}\int_0^{\hat s} \breve f(r,t) e^{-r}\,dr.
\end{equation}
for $\hat s, t>0$.  By the simple change of variables
\begin{equation}\label{eqD:change}
s= 1-e^{-\hat s}\,, \qquad U(s,t)= \breve f(\hat s,t)\,,
\end{equation}
one finds that \qref{eqD:BT1} for $\hat s\in(0,\infty)$, $t>0$, is equivalent to 
\begin{equation}\label{eqD:BTueq}
\D_t U(s,t) = -U^2-U + 2\int_0^1 U(sr,t)\,dr\,,
\end{equation}
for $s\in(0,1)$, $t>0$.  This equation has the same form that arises in the 
continuous-size case, as we discuss next.

 \subsection{Continuous-size distributions}

For clusters of any real size $x>0$, the size distribution at time $t$ 
is characterized by a measure $F_t$, whose distribution function we denote
using the same symbol:
\[
F_t(x) = \int_{(0,x]} F_t(dx).
\]
The measure $F_t$ evolves according to the following size-continuous coagulation-fragmentation equation, which we write in weak form.
One requires that for any  suitable test function $\varphi (x)$,
\begin{equation}\label{eq:CF2} \begin{split}
&\frac{d}{dt} \int_{\rplus} \varphi(x) \, F_t(dx) = 
\frac{1}{2} \int_{{\mathbb R}_+^2} \big( \varphi (x+y) - \varphi(x) - \varphi(y) \big) a(x,y) 
\, F_t(dx) \, F_t(dy) 
\\ & \quad 
- \frac{1}{2} \int_{{\mathbb R}_+} \Big( \int_0^x \big( \varphi (x) - \varphi(y) - \varphi(x-y) \big) \,  b(y,x-y) \, dy \, \Big)  F_t(dx) . 
\end{split}\end{equation}
The specific rate coefficients that we study correspond to constant coagulation rates
and constant overall binary fragmentation rates with uniform distribution of fragments:
\begin{eqnarray}
&&\hspace{-1cm}
a(x,y) = A\,, \qquad b(x,y) = \frac{B}{x+y}, 
\qquad A=B=2.\label{eq:rates_niwa_cont}
\end{eqnarray}
(Again, by scaling one can achieve any $A,B>0$.) 
We refer to the coagulation-fragmentation equations 
\qref{eq:CF2} with these coefficients as 
\textbf{Model C} (C for continuous size).

For size distributions with density, written as $F_t(dx)=f(x,t)\,dx$,
Model C is written formally in strong form as follows:
\begin{eqnarray}
&&\hspace{-1.5cm}
\D_t f(x,t) = Q_a(f)(x,t) + Q_b(f)(x,t) ,
\label{eq:CF3_Niwa_11}\\
&&\hspace{-1.5cm}
Q_a(f)(x,t) = \int_0^x \, f(y,t) \, f(x-y,t) \, dy -   2 f(x,t) \, \int_0^\infty f(y,t) \, dy , 
\label{eq:CF4_Niwa_11} \\
&&\hspace{-1.5cm}
Q_b(f)(x,t) = -f(x,t)  +  2 \int_x^\infty \frac{f(y,t)}{y} \, dy . 
\label{eq:CF5_Niwa_11}
\end{eqnarray}
Importantly, Model C has a scaling invariance involving dilation of size.
If $F_t(x)$ is any solution and $\lambda>0$, then
\begin{equation}\label{scaleC:F}
\hat F_t(x) := F_t(\lambda x)
\end{equation}
is also a solution. 

When we take as test function $\varphi(x)=1-e^{-sx}$, we find that the Bernstein transform
of $F_t$, defined by 
\begin{equation}
\uu(s,t)=\breve F_t(s) = \int_{\rplus} (1-e^{-sx})\,F_t(dx)\,,
\end{equation} 
satisfies
\begin{equation}
\D_t\uu(s,t) = -\uu^2 -\uu + 2 \int_0^1 \uu(sr,t)\,dr\,.
\label{eqC:Bernstein}
\end{equation}
This equation has exactly the same form as \eqref{eqD:BTueq}.

According to the well-posedness result for Model C established in \cite[Thm.~6.1]{DLP2017},
given any initial $F_0\in \calM_+(0,\infty)$ 
(the set of nonnegative finite measures on $(0,\infty)$), Model C admits
a unique narrowly continuous map $t\mapsto F_t\in \calM_+(0,\infty)$ 
that satisfies \eqref{eq:CF2} for all continuous $\varphi$ on $[0,\infty]$. 
In particular, \eqref{eqC:Bernstein} holds for all $s\in[0,\infty]$.
For $s=\infty$ in particular this means that the zeroth moment 
$m_0(t)=U(\infty,t)$ satisfies the logistic equation
\begin{equation}\label{e:m0}
\D_t m_0(t) = -m_0(t)^2 + m_0(t) \,,
\end{equation}
whence $m_0(t)\to1$ as $t\to\infty$.

\section{Preliminaries}\label{sec:prelim}

All of our main results concern the behavior of solutions of Models C and D
having power-law tails and infinite first moment, 
and the analysis involves the behavior of
their Bernstein transforms.  Hence, before we state our main results it is useful to recall some  basic definitions and results on Bernstein functions and transforms.

A function $g\colon(0,\infty)\to\R$ is \textit{completely monotone}
if it is infinitely differentiable and its derivatives satisfy $(-1)^ng\supn(x)\ge0$
for all real $x>0$ and integer $n\ge0$. By Bernstein's theorem, $g$ is completely
monotone if and only if it is the Laplace transform of some (Radon) measure on 
$[0,\infty)$. 

\begin{definition}
A function $U\colon (0,\infty)\to \R$ is a \textit{Bernstein function} if it is 
infinitely differentiable, nonnegative, and its derivative $U'$ is completely monotone.
\end{definition}

The main representation theorem for these functions 
\cite[Thm. 3.2]{Schilling_etal_Bernstein}
says that 
 a function $U\colon(0,\infty)\to\R$ is a Bernstein function if and only if
it has the representation
\begin{equation}\label{def:Btransform}
U(s) = a_0s+a_\infty+\int_{(0,\infty)} (1-e^{-sx})\,F(dx)\,, \quad s\in(0,\infty),
\end{equation}
where $a_0$, $a_\infty\ge0$ and $F$ is a measure satisfying
$\int_{(0,\infty)} (x\wedge1)F(dx)<\infty$. 
In particular, the triple $(a_0,a_\infty,F)$ uniquely determines $U$ and vice versa.

We point out that $U$ determines $a_0$ and $a_\infty$ via the relations
\begin{equation}\label{eq:a0ainfty}
a_0 = \lim_{s\to\infty} \frac{U(s)}s\,, \qquad
a_\infty = U(0^+)=\lim_{s\to0}U(s)\, .
\end{equation}

Whenever \qref{def:Btransform} holds, we call $U$
the {\sl Bernstein transform} of the L\'evy triple
$(a_0,a_\infty,F)$.  If $a_0=a_\infty=0$, we call $U$ the Bernstein
transform of the L\'evy measure $F$, and write $U=\breve F$, in accordance with the definitions in section \ref{sec:CF_general}.

We will also make use of  
the theory of so-called  \textit{complete} Bernstein functions, 
as developed in \cite[Chap. 6]{Schilling_etal_Bernstein}:
\begin{theorem} \label{thm:CBF}
The following are equivalent.
\begin{itemize}
\item[(i)] The L\'evy measure $F$ in \qref{def:Btransform}
has a completely monotone density $g$, so that
\begin{equation}\label{eq:CBF}
U(s) = a_0s+a_\infty+\int_{(0,\infty)} (1-e^{-sx}) g(x)\,dx\,, \quad s\in(0,\infty).
\end{equation}
\item[(ii)] $U$ is a Bernstein function that admits a holomorphic extension to the cut plane
$\C\setminus(-\infty,0]$ satisfying $(\im s)\im U(s) \ge0$. 
\end{itemize}
\end{theorem}
In complex function theory,
a function holomorphic on the upper half of the complex plane that leaves
it invariant is called a \textit{Pick function} (alternatively a 
\textit{Herglotz} or \textit{Nevalinna} function).
Condition (ii) of the theorem above says simply that $U$ is a Pick function
analytic and nonnegative on $(0,\infty)$. 
Such functions are called \textit{complete Bernstein functions} in 
\cite{Schilling_etal_Bernstein}.

The power-law tail behavior of size
distributions is related to power-law behavior of Bernstein transforms
near the origin through use of Karamata's Tauberian theorem
\cite[Thm.~III.5.2]{Feller} and Lemma 3.3 of \cite{MP2004}. 
To explain, suppose a measure $F$ on $(0,\infty)$ has a density $f$
satisfying
\begin{equation}\label{a:f}
f(x)\sim Ax^{-\alpha-1}\,,\quad x\to\infty.
\end{equation} 
Necessarily $\alpha\in(0,1]$ if $F$ is finite with infinite first moment. 
The derivative $\D_s\breve F$ of the Bernstein transform of $F$
is the Laplace transform of the measure with distribution function
\begin{equation}\label{a:F}
\int_0^x y\,F(dy) \sim \frac{A x^{1-\alpha} }{1-\alpha}
\end{equation}
for $\alpha\in(0,1)$.  By Karamata's theorem, this is equivalent to 
\begin{equation}\label{a:L}
\D_s\breve F(s) \sim 
\frac{A \Gamma(2-\alpha)}{1-\alpha}  s^{\alpha-1}, \quad s\to0.
\end{equation}
Then by Lemma 3.3 of \cite{MP2004} this is equivalent to  
\begin{equation}\label{a:U}
\breve F(s) \sim 
\frac{A \Gamma(2-\alpha)}{\alpha(1-\alpha)}  s^\alpha, \quad s\to0.
\end{equation}

\section{Main results}\label{sec:main}

The choice of coefficients in the asymptotic expressions below 
is made to simplify Bernstein transform calculations in the sequel. 
{In the following we denote by
\[
 F_t(x):=\int_{(0,x]} F_t(dx)
\]
the cumulative distribution function.}

\begin{theorem}\label{t:exist} (Self-similar solutions for Model C)
For each $\alpha\in(0,1)$ and $\lambda>0$,  
Model C admits a unique self-similar solution having the form
\begin{equation}\label{d:Fa}
F_t(x) = F_{\star\alpha}(\lambda xe^{-\beta t}),
\end{equation}
where $F_{\star\alpha}$ is a probability measure having the tail behavior
\begin{equation}
\int_0^x y F_{\star\alpha}(dy) \sim 
\frac{ \alpha} {\Gamma(2-\alpha)} x^{1-\alpha}\,,
\quad x\to\infty.
\end{equation}
For this solution, 
\begin{equation}\label{e:beta}
\beta = \frac{1-\alpha}{\alpha(1+\alpha)},
\end{equation}
and $F_{\star\alpha}$ has a completely monotone density
$f_{\star\alpha}$ 
having the following asymptotics:
\begin{equation}\label{e:faasym}
f_{\star\alpha}(x) \sim  
\begin{cases}
\displaystyle 
\frac{\alpha}{\Gamma(1-\alpha)} x^{-\alpha-1} & \quad x\to\infty\,,\\[8pt]
\displaystyle 
\frac{\hat c}{\Gamma(-\hat\alpha)} x^{\hat\alpha-1} & \quad x\to 0^+\,,
\end{cases}
\end{equation}
where the constants $\hat\alpha\in(0,1)$, $\hat c>0$ are as described in 
Lemma~\ref{lem:solbeta}. 
\end{theorem}

\begin{theorem}\label{t:mainC} (Large-time behavior for Model C with algebraic tails)
Suppose that the initial data for Model C satisfies
\begin{equation}\label{eqC:ics}
\int_0^x yF_0(dy) \sim 
\int_0^x yF_{\star\alpha}(\lambda\,dy) % \,,
\sim \frac{ \alpha\lambda^{-\alpha}} {\Gamma(2-\alpha)} x^{1-\alpha}\,,
\quad x\to\infty,
\end{equation}
where $\alpha\in(0,1)$, $\lambda>0$.  Then for every $x\in[0,\infty]$ we have
\begin{equation}\label{te:limitC}
F_t(x e^{\beta t}) \to F_{\star\alpha}(\lambda x) \,,
\quad t\to\infty.
\end{equation}
\end{theorem}

\begin{theorem}\label{t:mainD} 
(Large-time behavior for Model D with algebraic tails)
Suppose that the initial data for Model D satisfies
\begin{equation}\label{eqD:ics}
\sum_{1\le k\le x} k f_k(0) \sim
\int_0^x yF_{\star\alpha}(\lambda\,dy) % \,,
\sim \frac{ \alpha\lambda^{-\alpha}} {\Gamma(2-\alpha)} x^{1-\alpha}\,,
\quad x\to\infty,
\end{equation}
where $\alpha\in(0,1)$, $\lambda>0$.  Then for every $x\in[0,\infty]$ we have
\begin{equation}\label{te:limitD}
\sum_{1\le k\le xe^{\beta t}} f_k(t) 
\to F_{\star\alpha}(\lambda x) \,,
\quad t\to\infty.
\end{equation}
\end{theorem}

These convergence results relate to the notion of weak convergence
of measures on $(0,\infty)$ sometimes known as narrow convergence. 
Let $\calM_+(0,\infty)$ be the space of nonnegative finite (Radon) 
measures on $(0,\infty)$. Given $F, F_n\in \calM_+(0,\infty)$ for $n\in\N$,
we say $F_n$ converges to $F$ {\it narrowly} and write
$F_n \nto F$ if 
\[
\int_{(0,\infty)} g(x)\,F_n(dx) \to 
\int_{(0,\infty)} g(x)\,F(dx)
\]
for all functions $g\in C_b(0,\infty)$, the space of bounded continuous
functions on $(0,\infty)$.  The convergence statements \eqref{te:limitC}
and \eqref{te:limitD} correspond to the statement that
\[
\hat F_t(dx) \nto F_{\star\alpha}(\lambda\,dx),\quad t\to\infty
\]
where, respectively, 
\begin{equation}
\hat F_t(dx) = \begin{cases}
F_t(e^{\beta t}dx) & \mbox{for Model C,}\\
\sum_k f_k(t)\delta_{ke^{-\beta t}}(dx) & \mbox{for Model D.}
\end{cases}
\end{equation}
The proofs of \eqref{te:limitC} and \eqref{te:limitD} 
make use of the following result from \cite{DLP2017} {(cf. \cite[Proposition 3.6]{DLP2017})} that  
characterizes narrow convergence in terms of the Bernstein transform. 
\begin{proposition} \label{p:narrowfat}
Assume $F$, $F_n\in\calM_+(0,\infty)$ for $n\in\N$. 
Then the following are equivalent as $n\to\infty$.
\item[(i)] $F_n$ converges narrowly to $F$, i.e., $F_n\nto F$.
\item[(ii)] The Bernstein transforms $\breve F_n(s) \to \breve F(s)$, for each $s\in[0,\infty]$.
\item[(iii)] The Bernstein transforms $\breve F_n(s) \to \breve F(s)$, uniformly for $s\in(0,\infty)$. 
\end{proposition}

The proofs of our main results will proceed in stages as follows.
In section~\ref{s:ss} we identify the family of relevant 
self-similar solutions of equation~\eqref{eqC:Bernstein}.
The argument involves a phase plane analysis that does not yet establish 
that the profile function is actually a Bernstein function. 
In section~\ref{s:compare} we prove a comparison principle for the
nonlocal evolution equation \eqref{eqC:Bernstein}, then use this in
section~\ref{s:limits} to show that solutions of \eqref{eqC:Bernstein}
with initial data $U_0(s)\sim s^\alpha$ approach the corresponding 
self-similar form found in section~\ref{s:ss}.  
From this we deduce the self-similar profiles are limits of complete Bernstein
functions, hence they are Bernstein transforms themselves
of measures $F_{\star\alpha}$ having completely monotone densities, 
and the results of Theorems~\ref{t:mainC} and \ref{t:mainD} follow.  
The remaining properties of the profiles stated in Theorem~\ref{t:exist},
including complete monotonicity of densities and asymptotics for small
and large size, are established in sections~\ref{s:Pick} and \ref{s:Tauber}.

The results of Theorems~\ref{t:mainC} and \ref{t:mainD} show
that the long-time behavior of solutions with algebraic tails
depends upon the algebraic rate of decay. We recall that for the 
pure coagulation equation with constant rate kernel 
(corresponding to Model C without fragmentation), 
all domains of attraction for self-similiar solutions with 
algebraic tails were characterized in \cite{MP2004} by the condition
that initial data are regularly varying. Here
in Theorem~\ref{t:mainC}, for example, this would correspond
to the condition that the initial data satisfy
\[
\int_0^x y F_0(dy) \sim x^{1-\alpha} L(x)
\]
where $L$ is slowly varying at $\infty$.  
In the present context, however,
we do not know whether this more general condition 
is either sufficient or necessary for convergence to self-similar form.

\section{Self-similar scaling---necessary conditions}\label{s:ss}

We begin our analysis by finding the necessary forms
for any self-similar solution to equation \eqref{eqC:Bernstein}
that governs the Bernstein transform of solutions to Model C.

We look for self-similar solutions to \eqref{eqC:Bernstein} of the form
\[
\uu(s,t) = u(sX(t)),
\]
where $X(\cdot)$ is smooth with $X(t)\to\infty$ as $t\to\infty$.
Because in general $\uu(\infty,t)=m_0(t)\to1$ as $t\to\infty$, we require
$u(\infty)=1$. After substituting into \eqref{eqC:Bernstein}, we find that
for nontrivial solutions we must have 
\[
\beta :=X'(t)/X(t)
\] to be a positive constant independent
of $t$, and $u(z)$ must satisfy 
\begin{equation}
\beta z\D_z u + u^2+u = 2\int_0^1 u(zr)\,dr.
\label{eq:SSu1}
\end{equation}
With
\begin{equation}
v(z) = \int_0^1 u(zr)\,dr = \frac1z\int_0^z u(r)\,dr,
\end{equation}
the variables $(v(z),u(z))$ satisfy the ODE system 
\begin{align}
\label{e:ODEu}
\beta z\D_z u &= - u - u^2 + 2v\,,
\\ 
\label{e:ODEv}
z\D_z v &= u-v \,.
\end{align}
Under the change of variables $\tau = \log z$ we have $\D_\tau = z\D_z$ and this
system becomes autonomous. We seek a solution defined for $\tau\in\R$ satisfying
\[ (u,v) \to \begin{cases}(0,0) & \tau\to-\infty,\cr (1,1) & \tau\to+\infty,\end{cases}
\]
with both components increasing in $\tau$.
What is rather straightforward to check, is that the origin $(0,0)$
is a saddle point in the $(v,u)$ phase plane, and the region
\[
R = \{(u,v)\mid 0<\frac12(u+u^2)<v<u\}
\]
is positively invariant and contained in the unit square $[0,1]^2$. 
Inside this region both $u$ and $v$ increase with $\tau$. 
The unstable manifold at $(0,0)$ enters this region and must approach
the stable node $(1,1)$ as $\tau\to\infty$, satisfying $1\le dv/du \le \frac32$ 
asymptotically since the trajectory approaches from inside $R$.

This trajectory provides the following result.

\begin{lemma}\label{lem:solbeta}
Let $\beta>0$. 
Then, up to a dilation in $z$, there is a unique solution 
of \eqref{eq:SSu1} which is positive and increasing 
for $z\in(0,\infty)$ with $u(0)=0$ and $u(\infty)=1$, satisfying
\begin{align*}  
u(z)&\sim z^\alpha \qquad\mbox{as $z\to 0^+$},\\
1-u(z)&\sim \hat c z^{-\hat\alpha}
\quad\mbox{as $z\to \infty$},
\end{align*}
where $\alpha\in(0,1)$, $\hat\alpha\in(0,\frac13)$ are 
determined by the relations
\begin{equation}\label{e:eigsa}
\beta = \frac{1-\alpha}{\alpha(1+\alpha)}  
= \frac{1-3\hat\alpha}{\hat\alpha(1-\hat\alpha)}
\end{equation}
\end{lemma}
We note that 
the relations \eqref{e:eigsa} arise from the eigenvalue equations
\begin{equation}\label{e:eigs}
\left|\begin{matrix} -1 -\beta\alpha & 2\cr 1 & -1-\alpha 
\end{matrix}\right| = 0,
\qquad
\left|\begin{matrix} -3 +\beta\hat\alpha & 2\cr 1 & -1+\hat\alpha 
\end{matrix}\right| = 0\,.
\end{equation}

In what follows we let $u_\alpha$ denote the solution described by this lemma, 
noting that the relation between $\beta$ and $\alpha$
is monotone and given by \eqref{e:beta}.
The phase-plane argument above does not show that $u_\alpha$
is a Bernstein function, however.
Our plan is to show that in fact $u_\alpha$ is a complete Bernstein function (a Pick function),
by showing that it arises as the pointwise limit of rescaled solutions of
\eqref{eqC:Bernstein} which are complete Bernstein functions.  
Thus, our proof of the existence theorem \ref{t:exist}
will depend upon a proof of stability.

\section{Comparison principle}\label{s:compare}

Our next goal is to study the long-time dynamics 
of solutions of \eqref{eqC:Bernstein} with appropriate initial data.
For this purpose we develop a comparison principle 
showing that solutions of \eqref{eqC:Bernstein} preserve 
the ordering of the initial data on any interval of the form $[0,S]$.

Given $S>0$ and $u\in  C([0,S])$, %{\clb
{define an averaging operator $\calA$ by}
\begin{equation}\label{d:Av}
(\calA u)(s) = \int_0^1 u(sr)\,dr \,,\quad s\in[0,S].
\end{equation}
Then clearly $\calA $ is a linear contraction on $C([0,S])$, with 
\begin{equation}\label{e:A0}
(\calA u)(0)=u(0). 
\end{equation}
We recall that by Hardy's inequality,
\begin{equation}\label{e:hardy}
\left(\int_0^S |(\calA u)(s)|^2\,ds \right)^{1/2}
\le 2 \left(\int_0^S |u(s)|^2\,ds\right)^{1/2}.
\end{equation}
{%\clb
Indeed, due to Minkowski's inequality in integral form we have
\[
 \left(\int_0^S \left| \int_0^1 u(xr)\,dr\right|^2\,dx\right)^{1/2} \leq \int_0^1 \left(\int_0^S |u(sr)|^2\,ds \right)^{1/2}\,dr
\]
and thus 
\begin{align*}
 \left(\int_0^S |(\calA u)(s)|^2\,ds \right)^{1/2}&
 \leq \int_0^1 \left(\int_0^S |u(sr)|^2\,ds \right)^{1/2}\,dr \\
&\le \int_0^1 \frac{dr}{r^{1/2}} 
\left[\int_0^S |u(s)|^2\,ds \right]^{1/2}.
\end{align*}
}

\begin{proposition}\label{p:compare} 
Given $S, T>0$ suppose that $U, V\in C^1([0,T],C([0,S])$ 
have the following properties: 
\begin{itemize}
\item[(i)] $U(s,0)\ge V(s,0)$ for all $s\in[0,S]$,
\item[(ii)] for all $(s,t)\in[0,S]\times[0,T]$ the equations
\begin{eqnarray}
\D_t U + U^2 + U(s,t) = 2 \calA U + F \,,
\label{e:UF}\\
\D_t V + V^2 + V(s,t) = 2 \calA V + G \,,
\label{e:VG}\end{eqnarray}
hold, where $F\ge G$. 
\end{itemize}
Then $U\ge V$ everywhere in $[0,S]\times[0,T]$. 
 \end{proposition}

\begin{proof}
We write 
\[
w=U-V = w_+-w_- \quad\mbox{ where $w_+$, $w_-\ge0$.}
\]
Let $M\ge \max |U+V|$.  Subtracting \eqref{e:VG} from \eqref{e:UF} we find
\[
\D_t w + M |w| + w\ge 2\calA w + F-G \,.
\]
Because $w_\pm$ is Lipschitz in $t$, $w_+w_-=0$, and $\calA w_\pm\ge0$,
we can multiply by $-2w_- \le0$ and invoke \cite[Lemma~7.6]{GT}
to infer that the weak derivative
\[
\D_t(w_-^2) -2M w_-^2 \le 4 w_- \calA w_- \,.
\]
Integrating over $s\in[0,S]$ and using Hardy's inequality we find
\[
\D_t \int_0^S w_-(s)^2\,ds \le (8+2M)\int_0^S w_-(s)^2\,ds.
\]
Because $w_-(s,0)=0$, 
integrating in $t$ and using Gronwall's lemma concludes the proof
that $U\ge V$ in $[0,S]\times[0,T]$.
\end{proof}

\section{Convergence to equilibrium for initial data with power-law tails}
\label{s:limits}

We begin with a result for solutions of \eqref{eqC:Bernstein} 
that is suitable for use in treating both Model C and Model D.
\begin{proposition}\label{p:compareU}
Suppose $U(s,t)$ is any $C^1$ solution of \eqref{eqC:Bernstein} 
for $s\in[0,\bar s)$, $t\in[0,\infty)$,
and assume that its initial data satisfies
\begin{equation}\label{a:U0}
U_0(s) \sim s^\alpha \quad\mbox{as $s\to0^+$,}
\end{equation}
where $\alpha\in(0,1)$.  Then with $\beta$ given by \eqref{e:beta},
we have 
\begin{equation}
U(se^{-\beta t},t) \to u_\alpha(s) \quad\mbox{as $t\to\infty$, for all $s\in(0,\infty)$,}
\end{equation}
with uniform convergence for $s$ in any bounded subset of $(0,\infty)$,
where $u_\alpha$ is the self-similar profile $u$ described in Lemma~\ref{lem:solbeta}.
\end{proposition}

The proof is rather different from the 
proof of convergence to equilibrium for initial data with 
finite first moment, in section 7 of \cite{DLP2017}.
In the present case, the behavior of $U(s,t)$ globally in $t$ is 
determined by the local behavior of the initial data $U_0$ near $s=0$.

\begin{proof}
First, let $u_\alpha$ be given by Lemma~\ref{lem:solbeta},
and note that for any $c>0$ the function given by 
\[
V(s,t)=u_\alpha(cse^{\beta t})
\]
is a solution of \eqref{e:VG} with $G=0$. 
Second, it is not difficult to prove that 
\begin{equation}
 u_\alpha(cz)\to u_\alpha(z) \quad\mbox{as $c\to1$, uniformly for $z\in(0,\infty)$}.
\end{equation}

Now, let $S>0$ and let $\eps>0$.  
Choose $c<1<C$ such that
\begin{equation}
u_\alpha(cz)<u_\alpha(z)<u_\alpha(Cz)<u_\alpha(cz)+\eps
\quad\mbox{for all $z\in(0,\infty)$}.
\end{equation}
Due to the hypothesis \eqref{a:U0}, there exists $S_0=S_0(c,C)>0$ such that
\begin{equation}\label{a:cpm}
u_\alpha(cs) \le U(s,0) \le u_\alpha(C s) \qquad\mbox{for all $s\in[0,S_0]$.}
\end{equation}
Invoking the comparison principle in Proposition~\ref{p:compare} we infer that 
\begin{equation}\label{e:comp1}
u_\alpha(cse^{\beta t}) \le U(s,t) \le u_\alpha(C se^{\beta t}) 
\quad\mbox{for all $s\in[0,S_0]$, $t>0$}.
\end{equation}
Replacing $s\in[0,S_0]$ by $s e^{-\beta t}$ with $s\in[0,S_0e^{\beta t}]$, 
this gives
\begin{equation}\label{e:comp2}
u_\alpha(cs) \le U(se^{-\beta t},t) \le u_\alpha(C s) 
\quad\mbox{for all $s\in[0,S_0e^{\beta t}]$, $t>0$.}
\end{equation}
By consequence, whenever $S_0e^{\beta t}>S$ 
it follows that 
\[
|U(se^{-\beta t},t)-u_\alpha(s)|<\eps \quad\mbox{for all $s\in[0,S]$}.
\]
This finishes the proof.
\end{proof}

\begin{proof}[Proof of Theorem~\ref{t:mainC}]
Because of the dilation invariance of Model C, 
we may assume the initial data satisfies \eqref{eqC:ics}
with $\lambda=1$. By the discussion of \eqref{a:F}--\eqref{a:U}
we infer that 
\begin{equation}
U_0(s) = \int_0^\infty (1-e^{-sx}) F_0(dx) \sim s^\alpha\,,\quad s\to0.
\end{equation}
Next, we invoke Proposition~\ref{p:compareU} to deduce that 
\begin{equation}
U(se^{-\beta t},t) = \int_0^\infty (1-e^{-sx}) F_t(e^{\beta t}\,dx)
\to u_\alpha(s)
\end{equation}
for all $s\in[0,\infty)$. The limit also holds for $s=\infty$
as a consequence of the logistic equation {\eqref{e:m0}} for
$m_0(t)=U(\infty,t)$.
At this point we use the fact that the pointwise limit $u_\alpha(s)$
of the Bernstein functions $s\mapsto U(se^{-\beta t},t)$
is necessarily Bernstein \cite[Cor.~3.7, p.~20]{Schilling_etal_Bernstein} 
and the facts that
\[
\lim_{s\to0} u_\alpha(s) = 0, \qquad
\lim_{s\to\infty} u_\alpha(s) = 1\,,
\]
to infer the following
(cf.~\cite[Eq.~(3.3)]{DLP2017}). 

\begin{lemma}\label{lem:ualphaBernstein}
For any $\alpha\in(0,1)$, the function $u_\alpha$ described in
Lemma~\ref{lem:solbeta} is the Bernstein transform of a 
probability measure $F_{\star\alpha}$ on $(0,\infty)$, satisfying
\[
u_\alpha(s) = \int_0^\infty(1-e^{-sx})F_{\star\alpha}(dx)
\,,\quad s\in[0,\infty].
\]
\end{lemma}
Finally, we use Proposition~\ref{p:narrowfat} to 
infer the narrow convergence result
\begin{equation}\label{e:narrowC}
F_t(e^{\beta t}\,dx) \nto F_{\star\alpha}(dx)  \,,\quad t\to\infty,
\end{equation}
to conclude the proof of Theorem~\ref{t:mainC}, 
\end{proof}

\begin{proof}[Proof of Theorem~\ref{t:mainD}]
For Model D, the discussion of \eqref{a:F}--\eqref{a:U}
implies that the hypothesis \eqref{eqD:ics} on initial data
is equivalent to the condition 
\begin{equation}
\breve f(\hat s,0) \sim \lambda^{-\alpha}\hat s^\alpha\,, \quad\hat s\to0,
\end{equation}
on the Bernstein transform of the initial data.
Under the change of variables $s=1-e^{-\hat s}$ in \eqref{eqD:change}
this is evidently equivalent to 
\begin{equation}
U(s,0)\sim \lambda^{-\alpha}s^\alpha\,,\quad s\to0.
\end{equation}
As $U(s,t)$ is a solution of the dilation-invariant equation 
\eqref{eqD:BTueq}, so is the function $\hat U(s,t)=U(\lambda s,t)$
which satisfies $\hat U(s,0)\sim s^{\alpha}$, $s\to0$.
Invoking Proposition~\ref{p:compareU}, we deduce that
for all $s\in[0,\infty)$,
\begin{equation}\label{eqD:lim}
U(se^{-\beta t},t) \to u_\alpha(s/\lambda) \qquad\mbox{as $t\to\infty$}.
\end{equation}
Note that the left-hand side is well-defined only for $e^{\beta t}>s$.

We can now write
\begin{equation}
\breve f(\hat se^{-\beta t},t) = U(\bar s(\hat s,t)e^{-\beta t},t),
\end{equation}
where
$\bar s(\hat s,t) e^{-\beta t} = 1-\exp(-\hat s e^{-\beta t})$.
Then for any fixed $\hat s\in(0,\infty)$,
\begin{equation}
\bar s(\hat s,t)=\hat s+O(e^{-\beta t})
\quad\mbox{as $t\to\infty$.}
\end{equation}
Because the convergence in \eqref{eqD:lim} is uniform
for $s$ in bounded sets by Proposition~\ref{p:compareU},
it follows that for each $\hat s\in[0,\infty)$, 
\begin{equation}\label{eqD:flim}
\breve f(\hat se^{-\beta t},t)\to u_\alpha(\hat s/\lambda).
\end{equation}

Next we establish \eqref{eqD:flim} for $\hat s=\infty$,
recalling $\breve f(\infty,t)=m_0(f(t))$.
In the present case of Model D, the evolution equation
for $m_0(f(t))$ is not closed, and we formulate our
result as follows.

\begin{lemma}\label{l:m0limD} For any solution of Model D, 
$m_0(f(t)) \to 1$ as $t\to\infty$.
\end{lemma}

\begin{proof}
1. According to \cite[Thm.~12.1]{DLP2017},
the zeroth moment $m_0(f(t))=\breve f(\infty,t)$ is a smooth function 
of $t\in[0,\infty)$ that satisfies the inequality
\begin{equation}
\D_t m_0(f(t)) \le -m_0(f(t))^2 + m_0(f(t))\,,\quad t\ge0.
\end{equation}
We infer that for all $t\ge0$,
\begin{equation}\label{b:m0} 
m_0(f(t))\le  \frac1{1-e^{-t}}\,,
\end{equation}
as the right-hand size solves the logistic equation $y'=-y^2+y$ on 
$(0,\infty)$. Thus we infer
\begin{equation}\label{b:m0limsup}
\limsup_{t\to\infty} m_0(f(t)) \le 1.
\end{equation}
2. We claim $\liminf_{t\to\infty} m_0(f(t)) \ge 1$. For this we
use the result of Proposition~\ref{p:compareU}, with $U(s,t)$ 
for $0<s<1$ determined
from $\breve f(\hat s,t)$ by \eqref{eqD:change}. 
Choose $S>0$ such that 
$u_\alpha(S)>1-\eps$.
Then for $t$ sufficiently large we have
\[
m_0(f(t)) \ge U(S e^{-\beta t},t) >1-\eps.
\]
Hence $\liminf_{t\to\infty} m_0(f(t)) \ge 1$.
This finishes the proof of the Lemma.
\end{proof}

Now, because \eqref{eqD:flim} holds for all $s\in[0,\infty]$,
the desired conclusion of narrow convergence in Theorem~\ref{t:mainD}
follows by using Proposition~\ref{p:narrowfat}.
\end{proof}

\section{Pick properties of self-similar profiles}
\label{s:Pick}

\begin{lemma}\label{l:cmFa}
 For any $\alpha\in(0,1)$ the measure $F_{\star\alpha}$
of {Lemma } \ref{lem:ualphaBernstein} has a completely monotone density
$f_{\star\alpha}$, whose Bernstein transform 
is the function $u_\alpha$
described in Lemma~\ref{lem:solbeta}, i.e.,
\[
u_\alpha(s) = \int_0^\infty(1-e^{-sx})f_\alpha(x)\,dx, \quad
s\in[0,\infty].
\]
\end{lemma}

\begin{proof}
By Theorem 6.1(ii) of \cite{DLP2017}, if the initial data $F_0$ for 
Model C has a completely monotone density, then the solution $F_t$ has a 
completely monotone density for every $t\ge0$, 
with $F_t(dx)=f_t(x)\,dx$ where $f_t$ is completely monotone. 
By the representation theorem for complete Bernstein functions, 
this property is equivalent to saying that the Bernstein transform
$U(\cdot,t)=\breve F_t$ is a Pick function.

As dilates and pointwise limits of complete Bernstein functions 
are complete Bernstein functions \cite[Cor.~7.6]{Schilling_etal_Bernstein},
we infer directly from our Theorem \ref{t:mainC}
that for any $\alpha\in(0,1)$, the self-similar profile 
$u_\alpha$ is a complete Bernstein function. 
Therefore, its L\'evy measure $F_{\star\alpha}$ 
has a completely monotone density $f_\alpha$.
\end{proof}

\begin{remark}
An example of Pick-function initial data which satisfy the hypotheses of the convergence theorem is the following:
\begin{equation}
U_0(s) = s^\alpha = 
\frac{\alpha}{\Gamma(1-\alpha)}
\int_0^\infty (1-e^{-sx}) x^{-1-\alpha}\,dx\, 
\label{e:salphaBT}
\end{equation}
\end{remark}

\begin{remark}
We have no argument establishing the monotonicity of densities for model C 
that avoids use of the representation theorem for complete Bernstein functions.
It would be interesting to have such an argument.
\end{remark}

{\it Decomposition.}
A point which is interesting, but not essential to the main 
thrust of our analysis, is that we can sometimes `decompose' the 
Bernstein transforms $U(s,t)=\breve F_t(s)$ of solutions
of Model C, writing
\begin{equation}\label{d:Valp}
U(s,t) = V(s^\alpha,t),
\end{equation}
where $V(\cdot,t)$ itself is a complete Bernstein function.
By consider limits as $t\to\infty$,
this can be used to say something more about the self-similar
profiles $u_\alpha$.
\begin{proposition}\label{p:VaCB}
(a) Suppose $\alpha\in(0,1)$ and $U_0(s)=V_0(s^\alpha)$
where $V_0$ is completely Bernstein. Then for all $t\ge0$, 
\eqref{d:Valp} holds for the solution
of \eqref{eqC:Bernstein} with initial data $U_0$,
where $V(\cdot,t)$ is completely Bernstein.

(b) For each $\alpha\in(0,1)$, the Bernstein transform
$u_\alpha$ of the self-similar profiles of Lemma~\ref{lem:solbeta}
have the form
\begin{equation}\label{e:uaVa}
u_\alpha(s) = V_\alpha(s^\alpha)
\end{equation}
where $V_\alpha$ is completely Bernstein, having the representation
\begin{equation}\label{r:Va}
V_\alpha(s) = \int_0^\infty (1-e^{-sx})g_{\star\alpha}(x)\,dx
\end{equation}
for some completely monotone function $g_{\star\alpha}$.
\end{proposition}

\begin{proof}
To prove part (a), we define $V(\cdot,t)$ by \eqref{d:Valp} and
compute that 
\begin{gather} \label{e:DtV}
\D_t V(s,t) + V^2 + V = 2 A_\alpha V(s,t), 
\\
A_\alpha V(s,t) = \int_0^1 V(sr,t)\, d(r^{1/\alpha}).
\label{e:Aa}
\end{gather}
The implicit-explicit difference scheme used in 
\cite[Sec.~6]{DLP2017}
to solve \eqref{eqC:Bernstein} corresponds precisely here to
the difference scheme
\begin{equation}
\hat V_n(s) = V_n(s) + 2\Delta t\, A_\alpha V_n(s)\,,
\end{equation}
\begin{equation}
(1+\Delta t)V_{n+1}(s) + \Delta t\, V_{n+1}(s)^2 = \hat V_n(s)
\end{equation}
under the correspondence
\begin{equation}
U_n(s) = V_n(s^\alpha).
\end{equation}
Exactly as argued at the end of \cite[Sec.~6]{DLP2017}, if $V_n$
is completely Bernstein then so is $\hat V_n$ since 
complete Bernstein functions form a convex cone closed under
dilations and taking pointwise limits. Then $V_{n+1}$ is
completely Bernstein due to \cite[Prop.~3.4]{DLP2017}
(i.e., for the same reason $U_{n+1}$ is). 
Because of the fact 
that $U_n(s)\to U(s,t)$ as $\Delta t\to0$ with $n\Delta t\to t$.
which was shown in \cite{DLP2017}, we infer that similarly
$V_n(s)\to V(s,t)$, and hence $V(\cdot,t)$ is completely Bernstein.

Next we prove part (b).
From the convergence result of Proposition~\ref{p:compareU} %\ref{t:mainC} 
it follows that if $V_0(s)=U_0(s^{1/\alpha})\sim s$ as $s\to0$, 
then for all $s>0$,
\begin{equation}
V(se^{-\alpha\beta t},t)=U(s^{1/\alpha}e^{-\beta t},t)\to V_\alpha(s) \quad\mbox{as $t\to\infty$,} 
\end{equation}
where $V_\alpha$ is defined by \eqref{e:uaVa}.
By taking $V_0$ to be completely Bernstein and applying part (a),
we conclude $V_\alpha$ is completely Bernstein through
taking the pointwise limit.
\end{proof}

\begin{remark}
Formulae such as \eqref{e:uaVa}, involving the composition of
two Bernstein functions, 
are associated with the notion of subordination of probability measures,
as is discussed by Feller \cite[XIII.7]{Feller}.
See section~\ref{s:subord} below for further information.
\end{remark}

\begin{remark}
Equation \eqref{e:DtV} satisfied by $V(s,t)$ is close to one
satisfied by the Bernstein transform
of the solution of a system modeling coagulation with 
multiple-fragmentation \cite{Melzak1957,MLM1997,KKW2011}.
This system takes the following strong form analogous to 
\eqref{eq:CF3_Niwa_11}--\eqref{eq:CF5_Niwa_11}:
\begin{eqnarray}
&&\hspace{-1.5cm}
\D_t f(x,t) = Q_a(f)(x,t) + Q_b(f)(x,t) ,
\label{eq:CmF3_Niwa_11}\\
&&\hspace{-1.5cm}
Q_a(f)(x,t) = \int_0^x \, f(y,t) \, f(x-y,t) \, dy -   2 f(x,t) \, \int_0^\infty f(y,t) \, dy , 
\label{eq:CmF4_Niwa_11} \\
&&\hspace{-1.5cm}
Q_b(f)(x,t) = -f(x,t)  +   \int_x^\infty b(x|y) {f(y,t)} \, dy , 
\label{eq:CmF5_Niwa_11}
\end{eqnarray}
where
\begin{equation}\label{d:newb}
b(x| y) = (\gamma+2)\frac{x^\gamma}{y^{1+\gamma}} \,.
\qquad \gamma = \frac{1-\alpha}{\alpha} \,.
\end{equation}
The coefficient $\gamma+2$ is determined by the requirement that mass is conserved:
\[
1= \frac1y \int_0^y xb(x|y)\,dx =  (\gamma+2)\int_0^1 r^{\gamma+1}\,dr\,.
\]
A key calculation is that with $\vp_s(x)=1-e^{-sx}$,
\begin{align*}
\int_0^y \vp_s(x) b(x|y)\,dx 
 & =  \int_0^y \vp_s(x) (\gamma+2) \left(\frac xy\right)^{\gamma}\,\frac{dx}y
\\ & = \frac{\gamma+2}{\gamma+1}\int_0^1 \vp_s(ry)\,d(r^{\gamma+1})
\\ & = (\alpha+1)\int_0^1 \vp_s(ry)\,d(r^{1/\alpha})
\end{align*}
As a consequence, the Bernstein transform of a solution of
\eqref{eq:CmF3_Niwa_11}--\eqref{eq:CmF5_Niwa_11} should satisfy
\begin{gather} \label{e:DtVm}
\D_t V(s,t) + V^2 + V = (1+\alpha) A_\alpha V(s,t)\,. 
\end{gather}
{The coefficient $(1+\alpha)$ here differs from the factor 2
in \eqref{e:DtV},} and we see no way to scale
the $V$ in \eqref{d:Valp} to get exactly this 
coagulation--multiple-fragmentation model.

A last note is that the `number of clusters' produced from a cluster
of size $y$ by this fragmentation mechanism is calculated to be
\[
n(y) = \int_0^y b(x|y)\,dx = \frac{\gamma+2}{\gamma+1}=\alpha+1.
\]
\end{remark}

\section{Asymptotics of self-similar profiles}
\label{s:Tauber}

Here we complete the proof of Theorem~\ref{t:exist}, characterizing
self-similar solutions of Model C, by 
describing the asymptotic behavior of the self-similar size-distribution
profiles $f_{\star\alpha}$ in the limits of large and small size. 
This involves a Tauberian analysis based on the behavior of the Bernstein
transform $u_\alpha$ as described in Lemma~\ref{lem:solbeta}. 

\begin{proof}[Proof of Theorem~\ref{t:exist}] 
Given $\alpha\in(0,1)$, recall we know 
that for any self-similar solution of Model C as in \eqref{d:Fa},
the measure $F_{\star\alpha}(dx)$ must have Bernstein transform $u_\alpha(s)$
as described by Lemma~\ref{lem:solbeta}.  
That indeed the function $u_\alpha$ is the Bernstein transform
of a probability measure $F_{\star\alpha}$
follows from Lemma~\ref{lem:ualphaBernstein}, and the fact that 
$F_{\alpha\star}$ has a completely monotone density $f_{\star\alpha}$
was shown in Lemma~\ref{l:cmFa}.
It remains only to establish that $f_{\alpha\star}$ enjoys
the asymptotic properties stated in \eqref{e:faasym}.

From Lemma~\ref{lem:solbeta} we infer that as $z\to\infty$,
\[
1-u_\alpha(z) = \int_0^\infty e^{-zx} f_{\star\alpha}(x)\,dx 
\sim \hat c z^{-\hat\alpha}
\quad\mbox{as $z\to\infty$}.
\]
Recalling $\hat\alpha\in(0,\frac13)$, invoking the Tauberian theorem 
\cite[Thm.~XIII.5.3]{Feller} and the fact that $f_{\star\alpha}$ is monotone, 
from \cite[Thm.~XIII.5.4]{Feller} we infer
\begin{equation}\label{e:fa0}
f_{\star\alpha}(x) \sim 
 \frac{\hat c }{\Gamma(\hat\alpha)} x^{\hat\alpha-1}
\quad\mbox{as $x\to0$.}
\end{equation}

Next, from Lemma~\ref{lem:solbeta}, {\eqref{e:ODEu} and \eqref{e:beta}}   we infer that 
\[
\D_z u_\alpha(z) = \int_0^\infty e^{-zx} xf_{\star\alpha}(x)\,dx \sim \alpha z^{\alpha-1} 
\quad\mbox{as $z\to0$}.
\]
By Karamata's Tauberian theorem \cite[Thm.~XIII.5.2]{Feller} we deduce
\[
\int_0^x y f_{\star\alpha}(y)\,dy \sim  \frac{\alpha }{\Gamma(2-\alpha)} x^{1-\alpha}
\quad\mbox{as $x\to\infty$}.
\]
Although we do not know $y\mapsto y f_\alpha(y)$ is eventually monotone,
the selection argument used in the proof of 
\cite[Thm.~XIII.5.4]{Feller}
works without change, allowing us to infer that 
\begin{equation}
x f_{\star\alpha}(x) \sim  \frac{\alpha }{\Gamma(1-\alpha)}x^{-\alpha}
\quad\mbox{as $x\to\infty$}.
\end{equation}
This completes the proof of Theorem~\ref{t:exist}.
\end{proof}

\begin{remark} We note that in the limit $\alpha\to1$ we have $\beta\to0$
and $\hat\alpha\to\frac13$, and the power-law exponent 
$\hat\alpha-1\to-\frac23$. This recovers the exponent governing
the small-size behavior of the equilibrium distribution analyzed 
previously in \cite[Eq.~(1.6)]{DLP2017}.
\end{remark}

\begin{remark} By \eqref{e:uaVa}, 
\[
1-V_\alpha(z) = \int_0^\infty e^{-zx}g_{\star\alpha}(x)\,dx \sim 
\hat c z^{-\hat\alpha/\alpha}\,,
\]
hence by the same argument as that leading to \eqref{e:fa0}
we find
\begin{equation}\label{e:ga0}
g_{\star\alpha}(x) \sim \frac{\hat c}{\Gamma(\hat\alpha/\alpha)} 
x^{-1+\hat\alpha/\alpha}\quad \mbox{as $x\to0.$}
\end{equation}
We note that $\hat\alpha/\alpha<1$ for all $\alpha\in(0,1)$, because
the assumption $\hat\alpha=\alpha$ together with the relations
\eqref{e:eigsa} lead to a contradiction.
\end{remark}

\section{Series in fractional powers}
\label{s:series}

In this section we show that the self-similar profile in Lemma~\ref{lem:solbeta}
is expressed, for small $z>0$, in the form
\begin{equation}\label{e:ustarseries}
u_\alpha(z) = \sum_{n=1}^\infty (-1)^{n-1}c_n z^{\alpha n}\,,
\end{equation}
where the series converges for $z^\alpha\in(0,R_\alpha)$
for some positive but finite number $R_\alpha$, and the coefficient
sequence $\{c_n\}$ is positive with a rather nice structure.

By substituting the series expansion \eqref{e:ustarseries} into \eqref{eq:SSu1}
we find that $c_1=1$, 
and $c_n$ is necessarily determined recursively for $n\ge2$ by 
\begin{equation}\label{e:crecurs}
c_n = \frac1{a_n} \sum_{k=1}^{n-1} c_k c_{n-k}\,, 
\end{equation}
\begin{equation}
a_n = \beta\alpha n + 1 - \frac2{\alpha n+1} = 
\frac{1-\alpha}{1+\alpha}n + \frac{\alpha n-1}{\alpha n+1}.
\end{equation}
Because the relation \eqref{e:beta} implies that indeed
\begin{equation}\label{e:betaalpha2}
\beta\alpha+1 = \frac2{1+\alpha},
\end{equation}
plainly $a_1=0$ and $a_n$ increases with $n$, with $a_n>0$ for $n>1$. 

Recall that we know from Proposition~\ref{p:VaCB}
that $u_\alpha(s)=V_\alpha(s^\alpha)$ where
$V_\alpha$ is completely Bernstein.

\begin{proposition}
For each $\alpha\in(0,1)$, $V_\alpha$ is analytic in a neighborhood of 
$s=0$, given by the power series 
\[
V_\alpha(s) = \sum_{n=1}^\infty (-1)^{n-1}c_n s^{n}\,.
\]
This series has a positive radius of convergence $R_\alpha$ satisfying
\begin{equation}\label{e:Rbds}
\frac{1-\alpha}{1+\alpha} \le R_\alpha \le  a_2 <1,
\end{equation}
and coefficients that take the form
\begin{equation}\label{e:cng}
c_n = \gamma^\star_{n-1} R_\alpha^{1-n},
\end{equation}
where $(\gamma^\star_n)_{n\ge0}$ is a {\sl completely monotone sequence} 
with $\gamma^\star_0=1$.
\end{proposition}
\begin{proof} It suffices to prove the bounds 
on the radius of convergence and the representation formula \eqref{e:cng},
as the validity of equation \eqref{eq:SSu1} then follows by substitution.
By induction we will establish bounds on the 
radius of convergence of the power series
\begin{equation}\label{e:vstarseries}
v_\star(z) = \sum_{n=1}^\infty c_n z^n \,,
\end{equation}
which is evidently related to $V_\alpha$ by $V_\alpha(z)=-v_\star(-z)$.
Observe that the inequality $c_k \le {m}/{r^k}$ for $1\le k<n$ implies
\[
c_n \le 
\frac{n-1}{a_n} 
\frac{m^2}{r^n} \le  \frac{m}{r^n} \,,
\]
provided that 
\[
m\le \frac{a_n}{n-1} = \frac{1-\alpha}{1+\alpha} + \frac{2\alpha}{(1+\alpha)(1+\alpha n)} \,.
\]
By choosing
\[
m=r= 
\frac{1-\alpha}{1+\alpha}=\beta\alpha\,,
\]
we ensure $c_1=m/r$ and therefore $c_n\le r^{1-n}$ for all $n\ge1$, i.e.,
\begin{equation}
c_n \le \left( \frac{1+\alpha}{1-\alpha}\right)^{n-1} \,, \quad n=1,2,\ldots,
\end{equation}
whence 
\[
v_\star(z) \le \frac{rz}{r-z}<\infty \quad\mbox{for $0<z<r$}.
\]
In a similar way, the choice
\[
M=R = a_2 \ge \frac{a_n}{n-1} 
\]
for all $n\ge2$ ensures $c_n \ge {M}/{R^n}$ for all $n\ge2$, whence
\[
v_\star(z) \ge \frac{Rz}{R-z} \quad\mbox{for $0<z<R$}
\]
By consequence we infer the bounds in \eqref{e:Rbds} hold.

Now, because $V_\alpha$ is completely Bernstein, it is a Pick function
analytic on the positive half-line.
Hence, from what we have shown,
the function $v_\star$ is a Pick function analytic on $(-\infty,R_\star)$.
From this and Corollary 1 of \cite{LP2016}, it follows directly that the 
coefficients $c_n$ may be represented in the form \eqref{e:cng} 
where $\{\gamma_n\}_{n\ge0}$ is a completely monotone sequence
with $\gamma^\star_0=1$.
\end{proof}

\begin{remark}
In the limiting case $\beta=0$, $\alpha=1$, the coefficients $c_n$
reduce to the explicit form appearing in eq. (5.19) of \cite{DLP2017}.
Namely,
\[
c_n = A_n(3,1) = \frac1{3n+1} \binom{3n+1}{n}
\]
in terms of the Fuss-Catalan numbers defined by 
\[
A_n(p,r) = \frac1{pn+r}\binom{pn+r}{n} \,.
\]
This can be verified directly from the recursion formulae
in \eqref{e:crecurs}
by using a known identity for the Fuss-Catalan numbers
\cite[p.~148]{Riordan}.
\end{remark}

\begin{remark}
We are not aware of any combinatorial representation or interpretation 
of the coefficients $c_n(\alpha)$ for $\alpha\in(0,1)$, however.
\end{remark}

\begin{remark} {\it (Nature of the singularity at $R_\alpha$)} 
Numerical evidence suggests that for $0<\alpha<1$, the singularity
at $R_\alpha$ is a simple pole. 
 If true, this should imply that 
as $n\to\infty$, the coefficients $\gamma^\star_n \to \gamma^\star_\infty>0$,
and the completely monotone L\'evy density $g_{\star\alpha}(x)$ for the complete
Bernstein function $V_\alpha(z)$ has exponential decay at $\infty$, with
\[
g_{\star\alpha}(x) \sim C_\star e^{-R_\alpha x} \quad \mbox{as $x\to\infty$},
\]
where $C_\star>0$. 
\end{remark}

\section{A subordination formula}
\label{s:subord}
Here we use the subordination formulae from \cite[XIII.7(e)]{Feller}
as linearized in \cite[Remark 3.10]{ILP2015},
to describe a relation between the completely monotone L\'evy densities  
for the Bernstein functions $u_\alpha$ and $V_\alpha$.  Recall we have shown
\[
u_\alpha(z) = \int_0^\infty (1-e^{-zx})f_{\star\alpha}(x)\,dx 
= V_\alpha(z^\alpha),
\]
where $V_\alpha$ is a complete Bernstein function, with
\[
V_\alpha(z) = \int_0^\infty (1-e^{-zx})g_{\star\alpha}(x)\,dx\,,
\]
for some completely monotone function $g_{\star\alpha}$.
The complete Bernstein function $z^\alpha$ has power-law L\'evy measure 
\[
\nu_0(dx)=c_\alpha x^{-1-\alpha} dx\,,\qquad c_\alpha = \frac{\alpha}{\Gamma(1-\alpha)}.
\]
This is the jump measure for an $\alpha$-stable L\'evy process 
$\{Y_\tau\}_{\tau\ge0}$ (increasing in $\tau$)
whose time-$\tau$ transition kernel $Q_\tau(dy)$ has the Laplace transform
\[
\EE(e^{-qY_\tau}) = \int_0^\infty e^{-qy} Q_\tau(dy) = e^{-\tau q^\alpha}\,.
\]
Recalling
the subordination formula 
in the linearized form (3.20) from \cite{ILP2015},
we infer that the self-similar profile $f_{\star\alpha}$
may be expressed as
\[
f_{\star\alpha}(x) = %\nu_\alpha(dx) = 
\int_0^\infty Q_\tau(dx) 
g_{\star\alpha}(\tau)\,d\tau \,.
\]

We know that $Q_1(dy)=p_\alpha(y)\,dy$ where $p_\alpha$ is the
maximally skewed L\'evy-stable density from \cite[XVII.7]{Feller} given by 
\begin{equation}\label{e:palpha}
p_\alpha(x) = p(x;\alpha,-\alpha) = \frac{-1}{\pi x} 
\sum_{k=1}^\infty \frac{\Gamma(k\alpha+1)}{k!} (-x^{-\alpha})^k \sin k\pi\alpha\,.
\end{equation}
Then by scaling dual to $\exp(-\tau q^\alpha) = \exp(-(\tau^{1/\alpha}q)^\alpha)$, we find
\[
Q_\tau(dy) = 
p_\alpha\left(\frac{y}{\tau^{1/\alpha}}\right)\frac{dy}{\tau^{1/\alpha}}\,,
\]
and obtain the following.
\begin{proposition}
The self-similar profile $f_{\star\alpha}$ is related to 
the completely monotone L\'evy densitiy $g_{\star\alpha}$ 
of $V_\alpha$ by
\[
f_{\star\alpha}(x) = %\nu_\alpha(dx) = 
\int_0^\infty 
g_{\star\alpha}(\tau)
p_\alpha\left(\frac{x}{\tau^{1/\alpha}}\right)
\frac{d\tau }{ \tau^{1/\alpha}}
\]
\end{proposition}

We note that in the limit $\alpha\to1$ one has $p_\alpha(y)\,dy\to \delta_1$,
the delta mass at $1$, consistent with $g_{\star\alpha} \to f_{\star\alpha}$.
Moreover, note that from \eqref{e:palpha}
the large-$x$ behavior of the $\alpha$-stable density $p_\alpha$ is 
\begin{equation}
p_\alpha(x) \sim {\Gamma(1+\alpha)}\frac{\sin \pi\alpha}{\pi}
x^{-\alpha-1} 
\sim f_{\star\alpha}(x)\,,\quad x\to\infty\,,
\end{equation}
due to Euler's reflection formula for the $\Gamma$-function.
This is consistent with the fact that the Bernstein transform
of $p_\alpha$ is 
\[
\int_0^\infty(1-e^{-sx})p_\alpha(x)\,dx = 1-e^{-s^\alpha} \sim s^\alpha
\sim u_\alpha(s)\,, \quad s\to0.
\]

%%%%%%%%%%%%%%%%%%%%%%%%%%%%%%%%%%%%%%%%%(%%%%%%%%%%%
%%%%%%%%%%%%%%%%%%%%%%%%%%%%%%%%%%%%%%%%%%%%%%%%%%%%
%%%%%%%%%%%%%%%%%%%%%%%%%%%%%%%%%%%%%%%%%%%%%%%%%%%
%%%%%%%%%%%%%%%%%%%%%%%%%%%%%%%%%%%%%%%%%%%%%%%%%%%
\section*{Acknowledgements}
BN and RLP acknowledge support
from the Hausdorff Center for Mathematics and the CRC 1060 on {\it Mathematics of emergent effects}, Universit\"at Bonn. 
This material is based upon work supported by the National
Science Foundation under grants 
DMS 1514826 and 1812573 (JGL) and
DMS 1515400 and 1812609 (RLP),
partially supported by the Simons Foundation under grant 395796, %to RP.
by the Center for Nonlinear Analysis (CNA)
under National Science Foundation PIRE Grant no.\ OISE-0967140,
and by the NSF Research Network Grant no.\ RNMS11-07444 (KI-Net).
JGL and RLP acknowledge support from the 
Institut de Math\'ematiques, Universit\'e Paul Sabatier, Toulouse and the
Department of Mathematics, Imperial College
London under Nelder Fellowship awards. 

%\pagebreak

\bibliographystyle{plain}
%\bibliography{Niwabib}
%\bibliography{fatniwa-arxiv}

\ifdraft
{\bf Items to finish:}

\begin{itemize}
\item  [done] Edit acknowledgments
\item [done] (remark added at end of sec 3)
(Can we get necessary and/or sufficient conditions using regular variation?)
\item [done] Describe test sequences better in \eqref{eq:CF2_disc}? [Bounded as in \cite{DLP2017}]
\end{itemize}
\fi

\end{document}